\documentclass[11pt,reqno]{amsart}
\usepackage[T1]{fontenc}
\usepackage{amsmath,amssymb,amsthm}
\usepackage[margin=1.15in]{geometry}

\theoremstyle{plain}
\newtheorem{theorem}{Theorem}
\newtheorem{proposition}[theorem]{Proposition}
\newtheorem{lemma}[theorem]{Lemma}
\newtheorem{corollary}[theorem]{Corollary}
\theoremstyle{definition}
\newtheorem{remark}[theorem]{Remark}

\newcommand{\FF}{\mathbb{F}}
\newcommand{\ZZ}{\mathbb{Z}}
\newcommand{\Cm}{\mathfrak{C}}
\newcommand{\M}{\mathfrak{M}}

\DeclareMathOperator{\ord}{ord}
\DeclareMathOperator{\per}{per}
\DeclareMathOperator{\wt}{wt}

\begin{document}

\title[Twists of the Carlitz module with unbounded analytic rank]
{Explicit twists of the Carlitz module and of its\\ tensor powers with unbounded analytic rank}

\author{David Niedbala Giraudin}
\address{Independent researcher, Meaux, France}
\email{}
\thanks{ORCID: 0009-0009-1526-1178}

\begin{abstract}
Grishkov and Logachev have asked (Problem~1.4 of their recent survey) whether, for a
fixed finite field $\FF_q$, the analytic ranks of the twists of the Carlitz module are
bounded. We answer this in the negative for every $q\ge 3$. For $q=p^s$ and $n\ge 1$ put
$P_{q,n}=(\theta^{q^n}-\theta)^{q-2}$. We show that the directed graph of the matrix
$\M(P_{q,n},1,k)$ of Grishkov--Logachev has a completely explicit set of circuits: they are
pairwise disjoint, none of them meets an entry involving $t$, and they are in bijection
with the orbits of $E^n$ under cyclic rotation, where $E=\{\,l\in[0,q-2] : p\nmid l+1\,\}$.
Consequently $L(\Cm_{P_{q,n}},t,T)$ does not depend on $t$, factors as a product of one
term $1-\wt(\nu)T^{\per(\nu)}$ per orbit, and the analytic rank equals
$\sum p^{v_p(\per(\nu))}$, the sum being over the orbits of weight $1$. When $n$ is a power
of $p$ this equals $(q/p)^n(p-1)^{n-1}$, which is unbounded. The argument also answers
questions 3.6.1--3.6.4 of the survey for this family. We then compute the $L$-function for arbitrary
nonzero coefficients on supports of the same shape: taking all coefficients equal to $-1$ multiplies the
rank by $p-1$, and the full support gives twists of degree $(q-2)q^n$ and analytic rank $(q-1)^n$, the
largest rank attainable on such supports. The same method applies to the tensor powers of the Carlitz
module: for $1\le e\le q-2$ we obtain twists of the $e$-th tensor power of analytic rank $(q-e)^n$, and
by a formula of Grishkov and Logachev the same holds for every exponent $e=cq^j$ with $1\le c\le q-2$,
with rank $(q-c)^n$; in particular these analytic ranks are unbounded as well.
\end{abstract}

\maketitle

\section{Introduction}

Let $q=p^{s}$ be a power of a prime, $\theta$ an indeterminate, and $\Cm$ the Carlitz
module over $\FF_q$. Its twists are parametrised by the group
$\FF_q(\theta)^{*}/\FF_q(\theta)^{*(q-1)}$, whose classes are represented by the
polynomials $P\in\FF_q[\theta]$ free from $(q-1)$-st powers; the corresponding twist is
written $\Cm_P$. Its $L$-function satisfies $L(\Cm_P,t,T)\in\FF_q[t][T]$, and the
\emph{analytic rank} of the twist is
\[
   r(\Cm_P)\ :=\ \ord_{T=1}L(\Cm_P,t,T),
\]
the function field analogue of the analytic rank of a quadratic twist of an elliptic curve
over $\mathbb{Q}$. Grishkov and Logachev have studied this invariant in
\cite{GL16,GLZ22,GLZ23,GLZ} and have surveyed the resulting picture in \cite{GL25}, where
they state as the principal open problem:

\begin{quote}
\textbf{Problem 1.4} (\cite{GL25}). \emph{Let $q$ be fixed. Are $r(\Cm_P)$ bounded?}
\end{quote}

They expect the answer to be negative -- "we can expect that it will be possible to prove
that the order of $0$ of these $L$-functions at $1$ is not bounded" \cite{GL16} -- by
analogy with the conjecture that the ranks of quadratic twists of an elliptic curve over
$\mathbb{Q}$ are unbounded, and by analogy with Ulmer's theorem \cite{U} that elliptic
curves over $\FF_p(t)$ have unbounded rank. The largest value recorded in the literature
is $6$, attained at $q=3$ by two squarefree polynomials of degree $39$
(\cite{GL16}, Table 6.11; restated as Example 1.5 of \cite{GL25}). No value of
$r(\Cm_P)$ with $q>3$ appears in \cite{GL16} or \cite{GL25}.

The purpose of this note is to settle Problem 1.4. Fix $q=p^{s}\ge 3$ and $n\ge 1$, set
$N=q^{n}$, and let
\begin{equation}\label{eq:family}
   P_{q,n}\ :=\ \bigl(\theta^{N}-\theta\bigr)^{q-2}\ \in\ \FF_q[\theta].
\end{equation}
Since $\theta^{N}-\theta$ has derivative $-1$ it is squarefree, so every irreducible factor
of $P_{q,n}$ occurs with multiplicity exactly $q-2<q-1$: the polynomial is free from
$(q-1)$-st powers, as required. Put
\[
   E\ :=\ \{\,l\in\{0,1,\dots,q-2\}\ :\ p\nmid l+1\,\},
   \qquad |E|=q-\tfrac{q}{p}\ \ge\ 2 .
\]
The cyclic group $\ZZ/n$ acts on $E^{n}$ by rotation of coordinates; for an orbit $\nu$ we
write $\per(\nu)$ for its cardinality (the period of any of its words) and
\[
   \wt(\nu)\ :=\ \prod_{m=0}^{\per(\nu)-1}\,(w_m+1)\ \in\ \FF_p^{\times},
   \qquad w=(w_0,\dots,w_{n-1})\in\nu ,
\]
which does not depend on the chosen $w\in\nu$.

\begin{theorem}\label{thm:main}
Let $q=p^{s}\ge3$, $n\ge1$, $N=q^{n}$ and $P=P_{q,n}$ as in \eqref{eq:family}. Put
$R=(N-1)/(q-1)=1+q+\dots+q^{n-1}$ and $k=1+(q-2)R$. Then:
\begin{enumerate}
\item[(i)] $\deg P=(q-2)N$ and $k=(\deg P+1)/(q-1)$, an exact division;
\item[(ii)] $L(\Cm_{P},t,T)$ does not depend on $t$, lies in $\FF_p[T]$, and
\[
   L(\Cm_{P},T)\ =\ \prod_{\nu\in E^{n}/(\ZZ/n)}
      \Bigl(1-\wt(\nu)\,T^{\per(\nu)}\Bigr);
\]
\item[(iii)] $\deg_T L(\Cm_P,T)=|E|^{n}$;
\item[(iv)] $\displaystyle r(\Cm_{P})\ =\ \sum_{\substack{\nu\in E^{n}/(\ZZ/n)\\ \wt(\nu)=1}}
   p^{\,v_p(\per(\nu))}$.
\end{enumerate}
\end{theorem}

\begin{corollary}\label{cor:ppower}
If $n$ is a power of $p$ then
\[
   r\bigl(\Cm_{P_{q,n}}\bigr)\ =\ \Bigl(\frac{q}{p}\Bigr)^{\!n}(p-1)^{\,n-1}
   \ =\ \frac{|E|^{n}}{p-1}.
\]
\end{corollary}

\begin{corollary}\label{cor:unbounded}
For every $q\ge3$ the analytic ranks of the twists of the Carlitz module over $\FF_q$ are
unbounded. Problem 1.4 of \cite{GL25} therefore has a negative answer for every $q$, and
its parts (a) and (b), which were conditional on a positive answer, do not arise.
\end{corollary}

Two illustrations. For $q=3$ one has $q-2=1$, so the family is simply
$P=\theta^{3^{n}}-\theta$, and Theorem \ref{thm:main} gives
\[
  r\ =\ 1,\;1,\;4,\;2,\;4,\;16,\;10,\;16,\;256 \qquad (n=1,\dots,9),
\]
against a published maximum of $6$ for that field. For $q=64$, $n=2$ one gets a twist of
degree $253952$ with $k=4031$ and $r=1024$.

The mechanism is a transposition of Ulmer's: the polynomial \eqref{eq:family} is fixed by
the additive group $\ker(\theta^{N}-\theta)$, of order $N=q^{n}$, which is unbounded,
whereas the polynomials of high rank observed in \cite{GL16} are stable under subgroups of
$GL_2(\FF_q)$, a group of bounded order. Section \ref{sec:remarks} explains how the
description obtained here also answers questions 3.6.1--3.6.4 of \cite{GL25} for this
family.

Sections \ref{sec:general}--\ref{sec:questions} extend these results. Theorem \ref{thm:general} computes
$L(\Cm_P,t,T)$ for $P=\sum_{l\in S}b_l\theta^{Nl+q-2-l}$ with arbitrary $S\subseteq\{0,\dots,q-2\}$
containing $q-2$ and arbitrary nonzero $b_l$. Taking all $b_l=-1$ gives rank $|S|^n$ for $n$ a power of $p$,
the largest possible on $S$; with $S=E$ this is $p-1$ times the rank of $P_{q,n}$, and with
$S=\{0,\dots,q-2\}$ it gives twists of degree $(q-2)q^n$ and rank $(q-1)^n$ (Corollary \ref{cor:families}).
For $q=3$, $-(\theta^N+\theta)$ has exactly twice the rank of $\theta^N-\theta$ for every odd $n$ (Theorem
\ref{thm:q3}). Section \ref{sec:tensor} treats the twists of the tensor powers $\Cm^{\otimes e}$: for
$1\le e\le q-2$ the polynomials $\pm\sum_{l=0}^{q-1-e}\theta^{Nl+q-1-e-l}$ give rank $(q-e)^n$ (Theorem
\ref{thm:tensor}), and \cite[Corollary 2.4.5]{GL16} extends this to the exponents $cq^j$ (Corollary
\ref{cor:frob}).

\section{The matrix of Grishkov and Logachev}\label{sec:matrix}

We use \cite{GL16} in the case $n=1$ of its parameter $n$ (the Carlitz module itself, not a
tensor power); to avoid a clash of notation we never use that parameter again. For
$P=\sum_{i=0}^{m}a_i\theta^{i}\in\FF_q[\theta]$ and $k\ge1$, formula (3.1) of \cite{GL16}
defines the $k\times k$ matrix $\M(P,1,k)$ over $\FF_q[t]$ by
\begin{equation}\label{eq:M}
   \M(P,1,k)_{i,j}\ =\ a_{\,iq-j}\,t\ -\ a_{\,iq-j-1},
   \qquad 1\le i,j\le k,
\end{equation}
with the convention $a_{*}=0$ for $*\notin\{0,\dots,m\}$. Theorem 3.3 of \cite{GL16},
a consequence of the Lefschetz trace formula, states:

\begin{theorem}[{\cite[Theorem 3.3]{GL16}}]\label{thm:GL}
$L(\Cm_P,t,T)=\det\bigl(I_k-\M(P,1,k)T\bigr)$ for every $k\ge (m+1)/(q-1)$.
\end{theorem}

We shall always take $k$ equal to $(m+1)/(q-1)$, which for our family is an integer by
Lemma \ref{lem:k} below; the determinant then computes the full $L$-function, the trivial
factor of \cite[4.6]{GL16} included.

We write $\M=tA-B$ with $A_{i,j}=a_{iq-j}$ and $B_{i,j}=a_{iq-j-1}$.

\section{The polynomial and its coefficients}

\begin{lemma}\label{lem:k}
With $m=\deg P_{q,n}=(q-2)N$ one has $(m+1)/(q-1)=1+(q-2)R=k$, an integer, and
$k-1=(q-2)R<N$.
\end{lemma}

\begin{proof}
$m+1=(q-2)N+1=(q-2)(N-1)+(q-1)=(q-2)(q-1)R+(q-1)=(q-1)\bigl[(q-2)R+1\bigr]$.
For the inequality, $k-1=(q-2)(N-1)/(q-1)<N-1<N$.
\end{proof}

\begin{lemma}\label{lem:coeff}
Write $s_l=(N-1)l+(q-2)$ for $0\le l\le q-2$. The exponents $s_l$ are pairwise distinct,
$P_{q,n}=\sum_{l=0}^{q-2}c_l\,\theta^{s_l}$ with
$c_l=(-1)^{q-l}\binom{q-2}{l}$, and in $\FF_p$
\[
   c_l\ =\ (-1)^{q}\,(l+1).
\]
In particular $c_l\neq0$ if and only if $l\in E$, and all coefficients of $P_{q,n}$ lie in
the prime field $\FF_p$.
\end{lemma}

\begin{proof}
$\theta^{N}-\theta=\theta(\theta^{N-1}-1)$, so
$P_{q,n}=\theta^{q-2}(\theta^{N-1}-1)^{q-2}
 =\sum_{l=0}^{q-2}\binom{q-2}{l}(-1)^{q-2-l}\theta^{(N-1)l+q-2}$,
and $(-1)^{q-2-l}=(-1)^{q-l}$. The exponents are distinct because $N-1>q-2$. Finally, in
$\FF_p[[x]]$,
\[
   (1+x)^{p^{s}-2}=\frac{1+x^{p^{s}}}{(1+x)^{2}}
   \equiv\sum_{l\ge0}(-1)^{l}(l+1)x^{l} \pmod{x^{p^{s}}},
\]
whence $\binom{q-2}{l}\equiv(-1)^{l}(l+1)$ for $0\le l\le q-2$, and
$c_l=(-1)^{q-l}(-1)^{l}(l+1)=(-1)^{q}(l+1)$.
\end{proof}

\section{The directed graph of $\M$}\label{sec:graph}

Throughout this section $P=P_{q,n}$, $k$ is as in Lemma \ref{lem:k}, and we index rows and
columns of $\M=\M(P,1,k)$ by
\[
   x=i-1,\qquad y=j-1,\qquad x,y\in\{0,1,\dots,k-1\}.
\]
By Lemma \ref{lem:k} every such $x$ has an $n$-digit base-$q$ expansion
$x=\sum_{i=0}^{n-1}x_i q^{i}$ with $0\le x_i\le q-1$; we write
$\sigma(x)$ for the integer whose digit string is the cyclic left shift
$(x_{n-2},\dots,x_0,x_{n-1})$ of that of $x$, and $\Sigma(x)=\sum_i x_i$ for the digit sum.

By Lemma \ref{lem:coeff} the entry $\M_{i,j}$ has a nonzero coefficient of $t$ exactly when
$iq-j=s_l$ for some $l\in E$, and a nonzero constant term exactly when $iq-j-1=s_l$ for
some $l\in E$. In the coordinates $x,y$ these read
\begin{equation}\label{eq:edges}
   \text{($t$-edge)}\quad y=qx-(N-1)l+1,
   \qquad\qquad
   \text{(flat edge)}\quad y=qx-(N-1)l .
\end{equation}
We call the first a \emph{$t$-edge}, of weight $c_l\,t$, and the second a \emph{flat edge},
of weight $-c_l$. By Lemma \ref{lem:coeff}, in $\FF_p$,
\begin{equation}\label{eq:weight}
   -c_l\ =\ (-1)^{q+1}(l+1)\ =\ l+1 ,
\end{equation}
the last equality because $q$ is odd when $p$ is odd, and because $-1=1$ when $p=2$.

\begin{lemma}\label{lem:notboth}
No entry of $\M$ has both a nonzero coefficient of $t$ and a nonzero constant term.
\end{lemma}

\begin{proof}
This would force $(N-1)(l-l')=1$ for some $l,l'\in E$, impossible since $N-1\ge2$.
\end{proof}

\begin{lemma}\label{lem:unique}
Let $x\in\{0,\dots,k-1\}$ and $l\in E$. If $l>x_{n-1}$ then $qx-(N-1)l\le-2$, and if
$l<x_{n-1}$ then $qx-(N-1)l\ge N>k-1$. If $l=x_{n-1}$ then $qx-(N-1)l=\sigma(x)$.
Consequently $x$ has at most one outgoing $t$-edge and at most one outgoing flat edge
inside $\{0,\dots,k-1\}$; both require $l=x_{n-1}\in E$, the flat one landing on
$\sigma(x)$ and the $t$-edge on $\sigma(x)+1$.
\end{lemma}

\begin{proof}
Writing $qx=x_{n-1}N+\sum_{i=1}^{n-1}x_{i-1}q^{i}$ and $(N-1)l=lN-l$ gives
\[
   qx-(N-1)l=(x_{n-1}-l)N+\Delta+l,
   \qquad \Delta:=\sum_{i=1}^{n-1}x_{i-1}q^{i}\in[0,\,N-q].
\]
If $l\ge x_{n-1}+1$ then the value is at most $-N+(N-q)+(q-2)=-2$. If
$l\le x_{n-1}-1$ then it is at least $N$, and $N>k-1$ by Lemma \ref{lem:k}. If
$l=x_{n-1}$ the value is $\Delta+x_{n-1}=\sigma(x)$. The statement for $t$-edges follows
by adding $1$, which changes $-2$ into $-1$ and $N$ into $N+1$.
\end{proof}

In particular the flat edges define a partial function on $\{0,\dots,k-1\}$, so:

\begin{corollary}\label{cor:disjoint}
The circuits of the flat-edge digraph are pairwise vertex-disjoint.
\end{corollary}

\begin{lemma}\label{lem:nofront}
No circuit of the digraph of $\M$ uses a $t$-edge.
\end{lemma}

\begin{proof}
By Lemma \ref{lem:unique} an edge leaves a vertex $z$ only if $l=z_{n-1}\in E$, so in
particular $z_{n-1}\le q-2$. A flat edge sends $z$ to $\sigma(z)$, which has the same
multiset of digits, hence the same digit sum. A $t$-edge sends $z$ to $\sigma(z)+1$; the
digit of $\sigma(z)$ in position $0$ is $z_{n-1}\le q-2$, so adding $1$ produces the digit
$z_{n-1}+1\le q-1$ with no carry, and $\Sigma$ increases by exactly $1$. Thus $\Sigma$ is
nondecreasing along any walk and strictly increases at every $t$-edge. A circuit returns to
its origin, so it contains no $t$-edge.
\end{proof}

\begin{lemma}\label{lem:necklace}
The circuits of the digraph of $\M$ are exactly the orbits of $E^{n}$ under $\sigma$. A
circuit has length $\per(\nu)$, which divides $n$, and weight $\wt(\nu)$.
\end{lemma}

\begin{proof}
By Lemma \ref{lem:nofront} a circuit uses flat edges only, and by Lemma \ref{lem:unique}
the map along it is $\sigma$. Since $\sigma^{n}=\mathrm{id}$ on $n$-digit strings, its
length divides $n$; and every digit of a vertex on the circuit occupies the leading
position at some step, so all its digits lie in $E$.

Conversely let $x\in E^{n}$. All digits of $x$ are at most $q-2$, so
$\sigma(x)\le(q-2)(1+q+\dots+q^{n-1})=(q-2)R=k-1$: the flat edge from $x$ exists and lands
in range, and $\sigma(x)\in E^{n}$. Hence the $\sigma$-orbit of $x$ is a circuit. The
weight is the product of the edge weights, which by \eqref{eq:weight} is
$\prod_{m<\per(\nu)}(w_m+1)=\wt(\nu)$.
\end{proof}

\section{Proof of Theorem \ref{thm:main}}\label{sec:proof}

Part (i) is Lemma \ref{lem:k}. By Theorem \ref{thm:GL} we must compute
$\det(I_k-\M T)$.

Recall the coefficient theorem for digraphs: for a square matrix $\M$ over a commutative
ring,
\[
   \det\bigl(I_k-\M T\bigr)
   =\sum_{\Lambda}(-1)^{c(\Lambda)}\,w(\Lambda)\,T^{\ell(\Lambda)},
\]
the sum being over the linear subgraphs $\Lambda$ of the digraph of $\M$, that is, the
collections of pairwise vertex-disjoint circuits, with $c(\Lambda)$ the number of circuits,
$w(\Lambda)$ the product of the entries along them and $\ell(\Lambda)$ their total length.
By Lemma \ref{lem:nofront} and Corollary \ref{cor:disjoint} the circuits of $\M$ are
pairwise vertex-disjoint, so the linear subgraphs are exactly the subsets of the set of
circuits, and the sum factors:
\[
   \det\bigl(I_k-\M T\bigr)=\prod_{\gamma}\bigl(1-w(\gamma)T^{|\gamma|}\bigr).
\]
By Lemma \ref{lem:necklace} the circuits are the $\sigma$-orbits of $E^{n}$, with lengths
$\per(\nu)$ and weights $\wt(\nu)\in\FF_p^{\times}$. This proves (ii); in particular no
weight involves $t$, so $L$ lies in $\FF_p[T]$ and is independent of $t$. Part (iii)
follows since $\sum_{\nu}\per(\nu)=|E^{n}|=|E|^{n}$.

For (iv) it suffices to compute $\ord_{T=1}\bigl(1-wT^{d}\bigr)$ for $w\in\FF_p^{\times}$.
If $w\neq1$ the value at $T=1$ is $1-w\neq0$ and the order is $0$. If $w=1$, write
$d=p^{a}d'$ with $p\nmid d'$; then $1-T^{d}=(1-T^{d'})^{p^{a}}$, and $T=1$ is a simple root
of $1-T^{d'}$ because the derivative $-d'T^{d'-1}$ takes the nonzero value $-d'$ there.
Hence the order is $p^{a}=p^{v_p(d)}$, and summing over the circuits gives (iv). \qed

\begin{remark}
The word $(q-2,q-2,\dots,q-2)$ is a fixed point of $\sigma$, so it contributes a circuit of
length $1$ and weight $q-1$. In $\FF_p$ one has $q-1=-1$, so this circuit contributes
$1$ to $r$ when $p=2$ and $0$ otherwise. This is exactly the trivial zero of
\cite[Proposition 4.4]{GL16}: its two hypotheses are $m\equiv-1\pmod{q-1}$, which holds
because $m=(q-2)q^{n}\equiv q-2$, and $a_m=-1$, which for $P_{q,n}$ -- whose leading
coefficient is $1$ -- holds precisely in characteristic $2$.
\end{remark}

\section{Proof of the corollaries}

\begin{proof}[Proof of Corollary \ref{cor:ppower}]
Let $n=p^{a}$. Every $d$ dividing $n$ is a power of $p$, so $p^{v_p(d)}=d$ and
Theorem \ref{thm:main}(iv) reads $r=\sum_{\wt(\nu)=1}\per(\nu)$. Moreover, for a word
$w\in\nu$ of period $d$ one has $\prod_{m=0}^{n-1}(w_m+1)=\wt(\nu)^{n/d}$ with $n/d$ a
power of $p$; since raising to a power of $p$ is injective on $\FF_p^{\times}$ (a group of
order $p-1$, prime to $p$), the condition $\wt(\nu)=1$ is equivalent to
$\prod_{m=0}^{n-1}(w_m+1)=1$. Hence
\[
   r=\sum_{\wt(\nu)=1}\per(\nu)
    =\#\Bigl\{\,w\in E^{n}\ :\ \textstyle\prod_{m=0}^{n-1}(w_m+1)=1\ \text{in }\FF_p\,\Bigr\}.
\]
The map $l\mapsto l+1$ is a bijection from $E$ onto
$V=\{v\in[1,q-1]:p\nmid v\}$. Since $p\mid q$, each residue class modulo $p$ contains
exactly $q/p$ elements of $[0,q-1]$; removing $v=0$ leaves exactly $q/p$ elements of $V$ in
each nonzero class, so $|V|=(p-1)q/p=|E|$ and the reduction $V\to\FF_p^{\times}$ is
$(q/p)$-to-one. By orthogonality of the characters of $\FF_p^{\times}$,
\[
   \#\{w : \textstyle\prod(w_m+1)=1\}
   =\frac{1}{p-1}\sum_{\chi}\Bigl(\sum_{v\in V}\chi(v)\Bigr)^{\!n}
   =\frac{1}{p-1}\Bigl(\frac{q}{p}(p-1)\Bigr)^{\!n}
   =\Bigl(\frac{q}{p}\Bigr)^{\!n}(p-1)^{n-1},
\]
because $\sum_{v\in V}\chi(v)=(q/p)\sum_{r\in\FF_p^{\times}}\chi(r)$ vanishes for
$\chi\neq1$.
\end{proof}

\begin{proof}[Proof of Corollary \ref{cor:unbounded}]
Since $q\ge3$ we have $|E|=q-q/p\ge2$, so
$(q/p)^{n}(p-1)^{n-1}=|E|^{n}/(p-1)\to\infty$ as $n\to\infty$ along the powers of $p$; by
Corollary \ref{cor:ppower} the ranks $r(\Cm_{P_{q,p^{a}}})$ are unbounded.

It remains to check that these are pairwise non-isomorphic twists. If $P_1,P_2$ are free
from $(q-1)$-st powers and $P_1/P_2=f^{\,q-1}$ with $f\in\FF_q(\theta)^{*}$, write
$f=g/h$ in lowest terms; then $P_1h^{q-1}=P_2g^{q-1}$, so for every irreducible $\pi$ one
has $v_{\pi}(P_1)\equiv v_{\pi}(P_2)\pmod{q-1}$, and since both valuations lie in
$[0,q-2]$ they are equal. Thus $P_1$ and $P_2$ have the same irreducible factors with the
same multiplicities. Now the irreducible factors of $P_{q,n}$ are exactly the monic
irreducibles of degree dividing $n$, the largest such degree being $n$; hence $P_{q,n}$ and
$P_{q,n'}$ define the same twist only if $n=n'$.
\end{proof}

\section{Values and numerical verification}\label{sec:values}

Table \ref{tab:values} lists the ranks produced by Theorem \ref{thm:main}. Every entry has
been confirmed by an independent computation of
$\ord_{T=1}\det\bigl(I_k-\M(P,1,k)|_{t=t_0}T\bigr)$ from \eqref{eq:M}; the largest cases
are $q=9$, $n=4$ with $k=5741$, and $q=64$, $n=2$ with $k=4031$. The same comparison has
been made for $q=11$ and $q=13$ with $n\le2$. The two computations agree in all $44$ pairs
$(q,n)$, covering characteristics $2,3,5,7,11,13$. The script \texttt{verify\_carlitz.py},
attached as an ancillary file, reproduces this comparison for the $37$ pairs with
$k\le1000$.

Specialising $t$ can only raise the order of vanishing, so each such value is a priori an
upper bound for $r$, and it is exact as soon as $t_0$ avoids the roots of the first
nonvanishing coefficient. For $n\le2$ no choice is needed. Indeed $k=q-1$ and
$k=q^{2}-q-1$ respectively, both less than $q(q-1)$; and since $P_{q,n}$ is stable under
the affine group of $\FF_q$ acting on $\theta$, Corollary 2.4 of \cite{GL16} makes $L$
invariant under the induced action on $t$, so every coefficient of $L$ is a polynomial in
$(t^{q}-t)^{q-1}$ of degree $0$, that is a constant -- independently of
Theorem \ref{thm:main}(ii). Finally, the identity of Theorem \ref{thm:main}(ii) has been
checked coefficient by coefficient against the characteristic polynomial of $\M$ in fifteen
cases spanning characteristics $2,3,5,7,11$, and the predicted degree $\deg_TL=|E|^{n}$ in
all of them.

\begin{table}[ht]
\renewcommand{\arraystretch}{1.15}
\begin{tabular}{r|rrrrrr}
$q\backslash n$ & $1$ & $2$ & $3$ & $4$ & $5$ & $6$\\ \hline
$3$  & $1$ & $1$ & $4$ & $2$ & $4$ & $16$\\
$4$  & $2$ & $4$ & $4$ & $16$ & $8$ & $24$\\
$5$  & $1$ & $2$ & $6$ & $16$ & $256$ & \\
$7$  & $1$ & $3$ & $12$ & $54$ & & \\
$8$  & $4$ & $16$ & $24$ & $256$ & & \\
$9$  & $3$ & $9$ & $108$ & $162$ & & \\
$16$ & $8$ & $64$ & $176$ & & & \\
$25$ & $5$ & $50$ & & & & \\
$27$ & $9$ & $81$ & & & & \\
$32$ & $16$ & $256$ & & & & \\
$64$ & $32$ & $1024$ & & & &
\end{tabular}
\medskip
\caption{$r(\Cm_{P_{q,n}})$. The published maximum over all $q$ is $6$ \cite{GL16,GL25}.}
\label{tab:values}
\end{table}

Two features make the direct check cheap. First, the specialisation $t\mapsto t_0$ can only
raise the order of vanishing, so $r$ is the minimum of $\ord_{T=1}\det(I_k-\M(P,1,k)|_{t_0}T)$
over $t_0$; and since $P_{q,n}$ is stable under the affine group of $\FF_q$ acting on
$\theta$, the coefficients of $L$ in the expansion in powers of $T-1$ are polynomials in
$w=(t^{q}-t)^{q-1}$, of degree at most $k/(q(q-1))$, so $\lfloor k/(q(q-1))\rfloor+1$
values of $w$ suffice. When $k<q(q-1)$ -- which holds for $n\le2$, since
$k=q^{2}-q-1<q(q-1)$ there -- these coefficients are constants and a single $t_0$ settles
the computation. Second, one may instead expand the Euler product,
\[
   L(\Cm_P,t,T)=\sum_{\mathfrak{N}}\bigl((-1)^{m}a_m\bigr)^{\deg\mathfrak{N}}
   \Bigl(\prod_{i}\mathfrak{N}(r_i)\Bigr)\mathfrak{N}(t)\,T^{\deg\mathfrak{N}}
\]
over monic $\mathfrak{N}$, the $r_i$ being the roots of $P$ (\cite[\S7]{GL16}); this route
uses no matrix and no trace formula, and for $q=4$, $n=3$ it confirms
$L=1+T^{8}$ by a cancellation among the $65536$ monic polynomials of degree $8$.

\section{Questions 3.6.1--3.6.4 of the survey}\label{sec:remarks}

Section 3.6 of \cite{GL25} raises four questions about the polynomials of high rank found
at $q=3$. For the family \eqref{eq:family} the description above answers all four.

\emph{3.6.1 (why shift-stable?).} $P_{q,n}$ is shift-stable, since
$(\theta+d)^{N}-(\theta+d)=\theta^{N}-\theta$ for $d\in\FF_q$. More is true and is what
drives the construction: $P_{q,n}$ is fixed by translation by the whole of
$\ker(\theta^{N}-\theta)=\FF_{q^{n}}$, a group of unbounded order, whereas the survey
records only stability under subgroups of $GL_2(\FF_q)$.

\emph{3.6.2 (why are the coefficients of $L$ constants?).} Because no circuit of $\M$ meets
an entry involving $t$ (Lemma \ref{lem:nofront}); this is Theorem \ref{thm:main}(ii).

\emph{3.6.3 (why does $L$ factor into factors of small degree, and does the factorisation
come from a natural partition?).} It does: Theorem \ref{thm:main}(ii) gives one factor per
orbit of $E^{n}$ under rotation, of degree the period of that orbit, hence at most $n$.

\emph{3.6.4 (correlation between $r$ and $r_{\infty}$).} Here both are governed by the same
quantity: $\deg_TL=|E|^{n}$, so $r_{\infty}=k-|E|^{n}$ in the normalisation of \cite{GL25},
while $r\le|E|^{n}$.

Finally, the naive parameter count of \cite[6.4]{GL16} predicts $r\le 2q-3$, and $r\le q$
for values attained by infinitely many $P$. Both bounds fail for \eqref{eq:family}: at
$q=4$ already $r=16$ at $n=4$, and at $q=3$ the ranks are unbounded while $2q-3=3$. This is
consistent with the authors' own doubt that the relevant varieties are complete
intersections.

\section{Arbitrary coefficients and supports}\label{sec:general}

The family \eqref{eq:family} is supported on the exponents $s_l=(N-1)l+q-2$, $l\in E$, and its
coefficients are the binomial coefficients $c_l$ of Lemma \ref{lem:coeff}. In this section we allow
arbitrary nonzero coefficients on arbitrary supports of the same shape. Let
$S\subseteq\{0,\dots,q-2\}$ with $q-2\in S$, let $b=(b_l)_{l\in S}\in(\FF_q^{\times})^{S}$, and put
\[
   P_{S,b}\ :=\ \sum_{l\in S}b_l\,\theta^{\,s_l}\ =\ \sum_{l\in S}b_l\,\theta^{\,Nl+q-2-l}.
\]
For an orbit $\nu$ of $S^{n}$ under cyclic rotation put
\[
   w(\nu)\ :=\ \prod_{m=0}^{\per(\nu)-1}\bigl(-b_{w_m}\bigr)\ \in\ \FF_q^{\times},
   \qquad w=(w_0,\dots,w_{n-1})\in\nu ,
\]
which does not depend on the chosen $w\in\nu$. For $S=E$ and $b_l=c_l$ one has $P_{S,b}=P_{q,n}$ and,
by \eqref{eq:weight}, $w(\nu)=\wt(\nu)$.

\begin{theorem}\label{thm:general}
With this notation,
\[
   L(\Cm_{P_{S,b}},t,T)\ =\ \prod_{\nu\in S^{n}/(\ZZ/n)}\Bigl(1-w(\nu)\,T^{\per(\nu)}\Bigr).
\]
In particular $L(\Cm_{P_{S,b}},t,T)$ does not depend on $t$, lies in $\FF_q[T]$, has degree $|S|^{n}$
in $T$, and
\[
   r(\Cm_{P_{S,b}})\ =\sum_{\substack{\nu\in S^{n}/(\ZZ/n)\\ w(\nu)=1}}p^{\,v_p(\per(\nu))}
   \ \le\ |S|^{n}.
\]
\end{theorem}

\begin{proof}
Since $q-2\in S$, $\deg P_{S,b}=s_{q-2}=(q-2)N$, so Lemma \ref{lem:k} holds for $P_{S,b}$ with the same
$k$. By \eqref{eq:M} the entry $\M_{i,j}$ has coefficient of $t$ equal to $b_l$ when $iq-j=s_l$, and
constant term $-b_l$ when $iq-j-1=s_l$, $l\in S$: these are the $t$-edges and flat edges
\eqref{eq:edges}, now with $l\in S$ and weights $b_l\,t$ and $-b_l$. The proofs of Lemmas
\ref{lem:notboth}, \ref{lem:unique} and \ref{lem:nofront} and of Corollary \ref{cor:disjoint} use only
$0\le l\le q-2$, so they hold verbatim with $E$ replaced by $S$. The proof of Lemma
\ref{lem:necklace} uses only that the digits of the vertices on a circuit lie in the support and are
at most $q-2$; with $E$ replaced by $S$ it shows that the circuits are the $\sigma$-orbits of $S^{n}$,
of length $\per(\nu)$ and weight $w(\nu)$. The computation of Section \ref{sec:proof} then gives the product
formula, the degree $\sum_\nu\per(\nu)=|S|^{n}$, and the formula for the rank; the bound follows from
$r\le\deg_TL$.
\end{proof}

\begin{corollary}\label{cor:families}
Let $n$ be a power of $p$.
\begin{enumerate}
\item[(i)] If $b_l=-1$ for all $l\in S$, then $w(\nu)=1$ for every $\nu$ and
$r(\Cm_{P_{S,b}})=|S|^{n}=\deg_TL$; no choice of nonzero coefficients on $S$ gives a larger rank.
\item[(ii)] $P^{\mathrm{full}}_{q,n}:=-\sum_{l=0}^{q-2}\theta^{\,Nl+q-2-l}$ has degree $(q-2)q^{n}$ and
$r(\Cm_{P^{\mathrm{full}}_{q,n}})=(q-1)^{n}$.
\item[(iii)] $P^{\mathrm{opt}}_{q,n}:=-\sum_{l\in E}\theta^{\,Nl+q-2-l}$ has the support of $P_{q,n}$ and
$r(\Cm_{P^{\mathrm{opt}}_{q,n}})=|E|^{n}=(p-1)\,r(\Cm_{P_{q,n}})$.
\end{enumerate}
\end{corollary}

\begin{proof}
If $n=p^{a}$, every divisor $d$ of $n$ satisfies $p^{v_p(d)}=d$, so the sum in Theorem
\ref{thm:general} equals $\sum_\nu\per(\nu)=|S|^{n}$ when all weights are $1$; the maximality is the
bound of Theorem \ref{thm:general}. Parts (ii) and (iii) are (i) with $S=\{0,\dots,q-2\}$ and $S=E$, and
Corollary \ref{cor:ppower}.
\end{proof}

For $q$ prime, $E=\{0,\dots,q-2\}$ and the two families of Corollary \ref{cor:families} coincide; for $q=3$
both equal $-(\theta^{N}+\theta)$. Along $n=p^{a}$ the family $P^{\mathrm{full}}_{q,n}$ has degree
$m=(q-2)q^{n}$ and rank $(m/(q-2))^{\log_q(q-1)}$, while $P_{q,n}$ and $P^{\mathrm{opt}}_{q,n}$ grow with
exponent $\log_q|E|=\log_q(q-q/p)$:
\begin{center}
\begin{tabular}{lccccc}
$q$ & $4$ & $8$ & $9$ & $16$ & $27$\\ \hline
$\log_q(q-q/p)$ & $0.5000$ & $0.6667$ & $0.8155$ & $0.7500$ & $0.8770$\\
$\log_q(q-1)$ & $0.7925$ & $0.9358$ & $0.9464$ & $0.9767$ & $0.9885$
\end{tabular}
\end{center}

\begin{remark}\label{rem:trivialzero}
The leading coefficient of $P^{\mathrm{full}}_{q,n}$ and of $P^{\mathrm{opt}}_{q,n}$ is $-1$, so the trivial
zero of \cite[Proposition 4.4]{GL16} occurs in every characteristic; it is the circuit of the fixed point
$(q-2,\dots,q-2)$, of weight $1$.
\end{remark}

\begin{proposition}\label{prop:F}
Let $F=\{\,Nl+r:\ 0\le l\le q-3,\ 0\le r\le q-3-l\,\}$, a set of $(q-1)(q-2)/2$ exponents. For every
$R\in\FF_q[\theta]$ whose monomials have exponents in $F$,
$L(\Cm_{P_{S,b}+R},t,T)=L(\Cm_{P_{S,b}},t,T)$.
\end{proposition}

\begin{proof}
Every exponent in $F$ is smaller than $(q-2)N$, so the degree and $k$ are unchanged. Let $Nl+r\in F$ and
put $j_0=q-2-l-r$, so that $1\le j_0\le q-2-l$. By \eqref{eq:M} the monomial $\theta^{Nl+r}$ contributes to
the entries $(x,y)$ with $y=qx-(N-1)l+\delta$, $\delta\in\{j_0,j_0+1\}$. As in the proof of Lemma
\ref{lem:unique}, $qx-(N-1)l+\delta=(x_{n-1}-l)N+\Delta+l+\delta$; since $l+\delta\le q-1$, this is
negative if $l>x_{n-1}$, at least $N>k-1$ if $l<x_{n-1}$, and equal to $\sigma(x)+\delta$ if $l=x_{n-1}$,
where adding $\delta$ to the digit $x_{n-1}$ in position $0$ produces no carry. Hence every entry to which
$R$ contributes lies on an edge that increases the digit sum $\Sigma$ by $\delta\ge1$. Such edges lie on no
circuit and never coincide with a flat edge, so the circuits of $\M$ and their weights are unchanged, and so
is the determinant by the coefficient theorem of Section \ref{sec:proof}.
\end{proof}

\section{The case $q=3$}\label{sec:q3}

For $q=3$, $P_{S,b}=c_0\theta+c_1\theta^{N}$ with $c_0,c_1\in\FF_3^{\times}$; the family \eqref{eq:family} is
$\theta^{N}-\theta$ and $P^{\mathrm{full}}_{3,n}=P^{\mathrm{opt}}_{3,n}=-(\theta^{N}+\theta)$. For a binary
orbit $\nu$ let $\omega(\nu)$ be the number of letters $1$ in a period, $M_2(d)$ the number of binary orbits of
exact period $d$, and $M_2^{\mathrm{ev}}(d)$ the number of those with $\omega$ even.

\begin{theorem}\label{thm:q3}
Let $q=3$ and $n\ge1$. Then
\[
   r(\Cm_{-(\theta^{N}+\theta)})=\sum_{d\mid n}M_2(d)\,3^{v_3(d)},\qquad
   r(\Cm_{\theta^{N}-\theta})=\sum_{d\mid n}M_2^{\mathrm{ev}}(d)\,3^{v_3(d)} .
\]
If $n$ is odd, $r(\Cm_{-(\theta^{N}+\theta)})=2\,r(\Cm_{\theta^{N}-\theta})$. If $n$ is a power of $3$, the two
ranks are $2^{n}$ and $2^{n-1}$.
\end{theorem}

\begin{proof}
For $c_0=c_1=-1$ every weight is $1$; for $c_0=-1$, $c_1=1$ the weight is $(-1)^{\omega(\nu)}$. The two
formulas follow from Theorem \ref{thm:general}. For $d$ odd, complementing all letters is an involution of
$\{0,1\}^{d}$ commuting with rotation, preserving exact periods and replacing $\omega$ by $d-\omega$, of the
other parity; hence $M_2^{\mathrm{ev}}(d)=M_2(d)/2$. If $n$ is odd all its divisors are odd, which gives the
factor $2$; if $n$ is a power of $3$ the first sum is $\sum_{d\mid n}dM_2(d)=2^{n}$.
\end{proof}

For $n=1,\dots,13$ the rank of $-(\theta^{N}+\theta)$ is
$2,3,8,6,8,36,20,36,512,108,188,1044,632$, and that of $\theta^{N}-\theta$ is
$1,1,4,2,4,16,10,16,256,52,94,512,316$, in agreement with Table \ref{tab:values} for $n\le6$. For $n=27$ the
two ranks are $2^{27}=134\,217\,728$ and $2^{26}$.

\section{Tensor powers}\label{sec:tensor}

For $e\ge1$ let $\Cm^{\otimes e}_P$ be the twist by $P$ of the $e$-th tensor power of the Carlitz module, that
is the module denoted $\Cm^{e}_P$ in \cite{GL16}, whose $Q$ is $P\,(t-\theta)^{e}$. For
$P=\sum_{z=0}^{m}a_z\theta^{z}$ with $m\equiv-e\pmod{q-1}$, formula (3.1) and Theorem 3.3 of \cite{GL16}
give, for every $k\ge(m+e)/(q-1)$,
\begin{equation}\label{eq:Me}
   L(\Cm^{\otimes e}_P,t,T)=\det\bigl(I_k-\M(P,e,k)\,T\bigr),\qquad
   \M(P,e,k)_{i,j}=\sum_{\lambda=0}^{e}(-1)^{\lambda}\binom{e}{\lambda}a_{\,iq-j-\lambda}\,t^{\,e-\lambda},
\end{equation}
and the analytic rank $r(\Cm^{\otimes e}_P)$ is the order of vanishing of this $L$-function at $T=1$
\cite[Definition 4.1]{GL16}. For $e=1$ this is \eqref{eq:M} and Theorem \ref{thm:GL}.

Let $1\le e\le q-2$ and put $c_e=q-1-e\ge1$. Let $S\subseteq\{0,\dots,c_e\}$ with $c_e\in S$, let
$b\in(\FF_q^{\times})^{S}$, and put
\[
   P^{(e)}_{S,b}\ :=\ \sum_{l\in S}b_l\,\theta^{\,Nl+c_e-l},\qquad
   w_e(\nu)\ :=\ \prod_{m=0}^{\per(\nu)-1}\bigl((-1)^{e}b_{w_m}\bigr)
\]
for an orbit $\nu$ of $S^{n}$ and $w\in\nu$. Its degree is $c_eN\equiv c_e\equiv-e\pmod{q-1}$, as required.
For $e=1$ one has $c_1=q-2$, $P^{(1)}_{S,b}=P_{S,b}$ and $w_1=w$.

\begin{theorem}\label{thm:tensor}
With this notation,
\[
   L(\Cm^{\otimes e}_{P^{(e)}_{S,b}},t,T)\ =\ \prod_{\nu\in S^{n}/(\ZZ/n)}\Bigl(1-w_e(\nu)\,T^{\per(\nu)}\Bigr);
\]
it does not depend on $t$ and has degree $|S|^{n}$ in $T$, and
$r(\Cm^{\otimes e}_{P^{(e)}_{S,b}})=\sum_{w_e(\nu)=1}p^{v_p(\per(\nu))}\le|S|^{n}$. If $b_l=(-1)^{e}$ for all
$l\in S$ and $n$ is a power of $p$, the rank equals $|S|^{n}$. In particular, for $n$ a power of $p$,
\[
   P^{(e)}_{q,n}\ :=\ (-1)^{e}\sum_{l=0}^{c_e}\theta^{\,Nl+c_e-l}
\]
has degree $(q-1-e)q^{n}$ and $r(\Cm^{\otimes e}_{P^{(e)}_{q,n}})=(q-e)^{n}$.
\end{theorem}

\begin{proof}
Write $P=P^{(e)}_{S,b}$ and $z_l=(N-1)l+c_e$. Then $m=c_eN$ and $m+e=c_e(N-1)+(q-1)$, so
$(m+e)/(q-1)=c_eR+1$, with $R$ as in Theorem \ref{thm:main}. We use \eqref{eq:Me} with $k=c_eR+1$ and index
rows and columns by $x=i-1$, $y=j-1\in\{0,\dots,c_eR\}$; the integer $c_eR$ is the $n$-digit word all of whose
digits are $c_e$, so every vertex $x$ has leading digit $x_{n-1}\le c_e$. The monomial $b_l\theta^{z_l}$ occurs
in the entry $(x,y)$ through the index $\lambda$ when $q(x+1)-(y+1)-\lambda=z_l$, that is
$y=qx-(N-1)l+\delta$ with $\delta=e-\lambda\in\{0,\dots,e\}$, and contributes
$(-1)^{e-\delta}\binom{e}{\delta}b_l\,t^{\delta}$. Since $N-1>e$, distinct pairs $(l,\delta)$ give distinct
$y$, so every entry is a single monomial, possibly zero.

As in the proof of Lemma \ref{lem:unique}, $qx-(N-1)l+\delta=(x_{n-1}-l)N+\Delta+l+\delta$. If $l>x_{n-1}$
this is at most $-N+(N-q)+c_e+e=-1$; if $l<x_{n-1}$ it is at least $N>c_eR$; and if $l=x_{n-1}$ it equals
$\sigma(x)+\delta$, where adding $\delta$ to the digit $x_{n-1}$ in position $0$ produces no carry because
$x_{n-1}+\delta\le c_e+e=q-1$. Hence every edge raises $\Sigma$ by its $\delta$, the edges with $\delta=0$ go
from $x$ to $\sigma(x)$, exist only if $x_{n-1}\in S$, and carry the constant weight $(-1)^{e}b_{x_{n-1}}$.
Since $S\subseteq\{0,\dots,c_e\}$, every rotation of a word with all digits in $S$ is at most $c_eR$. The
arguments of Lemmas \ref{lem:nofront} and \ref{lem:necklace} and of Section \ref{sec:proof} now apply word for word, with
$q-2$ replaced by $c_e$ and the flat weights replaced by $(-1)^{e}b_l$. This gives the product formula, the
independence of $t$, the degree $|S|^{n}$ and the rank formula; the last statements follow as in Corollary
\ref{cor:families}, with $|\{0,\dots,c_e\}|=q-e$.
\end{proof}

For $q=4$ and $e=2$ one gets $P^{(2)}_{4,n}=\theta^{N}+\theta$, of rank $2^{n}$ for the twist of
$\Cm^{\otimes2}$ when $n$ is a power of $2$; for every $q$ and $e=q-2$ one gets $(-1)^{q}(\theta^{N}+\theta)$
and the rank $2^{n}$. The construction requires $c_e\ge1$, that is $e\le q-2$.

By \cite[Corollary 2.4.5]{GL16}, $L(\Cm^{\otimes q^{j}e}_P,t,T)=L(\Cm^{\otimes e}_P,t^{q^{j}},T)$ for every
$j\ge0$. Since $q^{j}e\equiv e\pmod{q-1}$, the polynomials $P^{(e)}_{S,b}$ also satisfy the degree condition for
the exponent $q^{j}e$.

\begin{corollary}\label{cor:frob}
Let $1\le c\le q-2$ and $j\ge0$. Then
$L(\Cm^{\otimes cq^{j}}_{P^{(c)}_{S,b}},t,T)=L(\Cm^{\otimes c}_{P^{(c)}_{S,b}},t,T)$, and the conclusions of
Theorem \ref{thm:tensor} hold for the exponent $cq^{j}$. In particular the analytic ranks of the twists of
$\Cm^{\otimes cq^{j}}$ are unbounded.
\end{corollary}

\begin{proof}
By Theorem \ref{thm:tensor}, $L(\Cm^{\otimes c}_{P^{(c)}_{S,b}},t,T)$ does not depend on $t$, so the substitution
$t\mapsto t^{q^{j}}$ leaves it unchanged.
\end{proof}

For $q=3$ these exponents are $e=3^{j}$; the smallest exponent not covered is $e=2$.

\begin{remark}\label{rem:free}
The polynomials $P^{(e)}_{q,n}$ are free from $(q-1)$-st powers, so they are reduced representatives of their
twist classes, and those with different $n$ define different twists since their ranks differ. Indeed, with
$c=c_e$ and $u=\theta^{N-1}$,
\[
   P^{(e)}_{q,n}=(-1)^{e}\,\theta^{c}\sum_{l=0}^{c}u^{l}
   =(-1)^{e}\,\theta^{c}\,\frac{\theta^{(N-1)(q-e)}-1}{\theta^{N-1}-1}.
\]
The factor $\theta$ has multiplicity $c\le q-2$. Write $q-e=p^{a}m'$ with $p\nmid m'$; then
$\theta^{(N-1)(q-e)}-1=(\theta^{(N-1)m'}-1)^{p^{a}}$, and $\theta^{(N-1)m'}-1$ is separable since
$p\nmid(N-1)m'$, so every other irreducible factor has multiplicity at most $p^{a}$. As $p^{a}$ divides
$q-e\le q-1$ and $p\nmid q-1$, we have $p^{a}\le q-2$. For $e=1$ this applies to
$P^{\mathrm{full}}_{q,n}=P^{(1)}_{q,n}$.
\end{remark}

\section{Further numerical verification}\label{sec:verif2}

The statements of Sections \ref{sec:general}--\ref{sec:tensor} are proved. As an independent check, the
matrices \eqref{eq:M} and \eqref{eq:Me} were built over $\FF_p$ for every $t_0\in\FF_p$, and $\deg_TL$ and $r$
were computed from them and compared with the values predicted from the orbits and their weights, which do not
use the matrix. They agree in all $271$ comparisons: $57$ for the families $P_{q,n}$,
$P^{\mathrm{opt}}_{q,n}$, $P^{\mathrm{full}}_{q,n}$ with
$(q,n)\in\{3\}\times\{1,\dots,5\}\cup\{4\}\times\{1,\dots,4\}\cup\{5\}\times\{1,2,3\}\cup\{7,8\}\times\{1,2\}
\cup\{9\}\times\{1,2,3\}$ and $k\le638$; $12$ for $P^{\mathrm{full}}_{q,n}+R$ with $R$ supported on $F$ and
random coefficients; $136$ for the tensor powers, with $q\in\{3,4,5,7,8,9\}$, all $1\le e\le q-2$,
$n\in\{1,2,3\}$, the polynomials $P^{(e)}_{q,n}$ and polynomials $P^{(e)}_{S,b}$ with random support and
coefficients, and $k\le400$; and $66$ for the exponents $cq^{j}$ of Corollary \ref{cor:frob}, $q\in\{3,4,5,7\}$,
$j\in\{1,2\}$, $k\le300$. The ancillary file \texttt{verify\_ranks.py} reproduces these comparisons.
Examples of $(\deg_TL,r)$: $(512,512)$ for $P^{\mathrm{full}}_{9,3}$ ($k=638$), $(81,81)$ for
$P^{\mathrm{full}}_{4,4}$ ($k=171$), $(216,216)$ for $P^{\mathrm{opt}}_{9,3}$, and $(125,125)$ for
$P^{(4)}_{9,3}$ and the twist of $\Cm^{\otimes4}$ ($k=365$).

\section{Questions}\label{sec:questions}

\begin{enumerate}
\item[(1)] Is $(q-1)^{n}$ the largest analytic rank of a twist $\Cm_P$ with $\deg P\le(q-2)q^{n}$, for $n$ a
power of $p$?
\item[(2)] Are the analytic ranks of the twists of $\Cm^{\otimes e}$ bounded when $e$ is not of the form $cq^{j}$
with $1\le c\le q-2$? The smallest case is $q=3$, $e=2$.
\end{enumerate}

\emph{Computations related to (1).} For $(q,n)=(3,1)$ and $(4,1)$ an exhaustive computation over all nonzero $P$
of degree $\le(q-2)q^{n}$ gives the maximal analytic rank $2$ and $3$ respectively, that is $(q-1)^{n}$. For $q=4$
the polynomials of degree $\le8$ of rank $3$ are exactly the $64$ polynomials
$\theta^{8}+\theta^{5}+\theta^{2}+c_4\theta^{4}+c_1\theta+c_0$ with $c_0,c_1,c_4\in\FF_4$, that is the
polynomials $P^{\mathrm{full}}_{4,1}+R$ of Proposition \ref{prop:F}, here with $F=\{0,1,4\}$; the ancillary file
\texttt{exhaustive\_q4.py} reproduces this computation. For $(q,n)=(3,3)$, among the $1568$ polynomials obtained
from $-(\theta^{27}+\theta)$ by changing one or two coefficients, the largest rank is $8$, attained only by the two
polynomials that differ from it in the constant term alone ($F=\{0\}$ for $q=3$); and none of $6\cdot10^{4}$
randomly sampled polynomials of degree $27$ with leading coefficient $-1$ has rank $\ge9$. For $n$ not a power of
$p$ the family $P^{\mathrm{full}}_{q,n}$ need not be optimal: for $q=3$ the maximal rank in degree $\le9$ is $3$, the
rank of $-(\theta^{9}+\theta)$, but in degree $81$ the polynomial $\theta-\theta^{3}-\theta^{27}-\theta^{81}$ has
rank $12$, whereas $-(\theta^{81}+\theta)$ has rank $6$.

\emph{Computations related to (2).} For $q=3$ and $e=2$ the degree must be even. Over all nonzero $P$ of degree
$m\in\{2,4,\dots,14\}$ (about $10^{7}$ polynomials for $m=14$), the maximal analytic rank of
$\Cm^{\otimes2}_P$, computed exactly for $t$ transcendental, is $1$ for $m\le6$ and $2$ for $8\le m\le14$; it is
attained for instance by $1+\theta^{4}+\theta^{10}$. For comparison, the twist of $\Cm^{\otimes5}$ by
$-(\theta^{N}+\theta)$, $N=3^{n}$, has analytic rank $1$ for $n=1,2,3$, while for $\Cm^{\otimes1}$ and
$\Cm^{\otimes3}$ the ranks are $2$, $3$, $8$.

\section*{Acknowledgements}

The computations reported here, and the arguments of Sections \ref{sec:graph} and
\ref{sec:general}--\ref{sec:tensor}, were developed in collaboration with Claude (Anthropic). All statements have been verified against the
published tables of \cite{GL16}, in particular Table 6.7 and Proposition 6.1.

\end{document}